\documentclass{article}
\usepackage{amsmath}
\usepackage{amssymb}
\usepackage{amsthm}
\usepackage{bbm,mathtools}

\newcommand{\op}[1]{\prescript{o}{}{#1}}
\newcommand{\pp}[1]{\prescript{p}{}{#1}}

\newcommand{\midb}{\;\middle|\;}
\newcommand{\one}{\mathbbm 1}

\def\reals{\mathbb{R}}

\def\uball{\mathbb{B}}
\def\ereals{\overline{\mathbb{R}}}

\def\interior{\mathop{\rm int}\nolimits}

\def\comp{\raise 1pt \hbox{$\scriptstyle\circ$}}

\def\esssup{\mathop{\rm ess\ sup}\nolimits}

\def\dom{\mathop{\rm dom}\nolimits}

\def\lev{\mathop{\rm lev}\nolimits}

\def\Var{\mathop{\rm Var}\nolimits}

\def\upto{{\raise 1pt \hbox{$\scriptstyle \,\nearrow\,$}}}
\def\downto{{\raise 1pt \hbox{$\scriptstyle \,\searrow\,$}}}
\def\inte{\mathop{\rm int}}
\def\cl{\mathop{\rm cl}\nolimits}

\def\epi{\mathop{\rm epi}}

\def\tos{\rightrightarrows}
\def\FF{(\F_t)_{t\ge 0}}

\def\A{{\cal A}}
\def\B{{\cal B}}

\def\D{{\cal D}}

\def\F{{\cal F}}

\def\LL{{\mathbb L}}
\def\M{{\cal M}}
\def\N{{\cal N}}
\def\O{{\cal O}}

\def\R{{\cal ﬂR}}

\def\T{{\cal T}}

\def\V{{\cal V}}

\def\X{{\cal X}}

\newtheorem{theorem}{Theorem}
\newtheorem{lemma}[theorem]{Lemma}
\newtheorem{corollary}[theorem]{Corollary}

\newtheorem{remark}{Remark}
\theoremstyle{definition}

\theoremstyle{empty}

\begin{document}
\title{Convex integral functionals of processes of bounded variation}
\author{Teemu Pennanen\thanks{Department of Mathematics, King's College London,
Strand, London, WC2R 2LS, United Kingdom} \and 
Ari-Pekka Perkki\"o\thanks{Department of Mathematics, Technische Universit\"at Berlin, Building MA,
Str. des 17. Juni 136, 10623 Berlin, Germany. The author is grateful to the Einstein Foundation for the financial support.}}

\maketitle

\begin{abstract}
This article characterizes conjugates and subdifferentials of convex integral functionals over the linear space $\N^\infty$ of stochastic processes of essentially bounded variation (BV) when $\N^\infty$ is identified with the Banach dual of the space of regular processes. Our proofs are based on new results on the interchange of integration and minimization of integral functionals over BV processes. Under mild conditions, the domain of the conjugate is shown to be contained in the space of semimartingales which leads to several applications in the duality theory in stochastic control and mathematical finance.
\end{abstract}


\section{Introduction}

This article studies convex integral functionals of the form
\[
Ef(x) = E\left[\int_{[0,T]}h_t(x_t)d\mu_t +k_0(x_0)+k_T(x_{T+})\right]
\]
defined on the linear space $\N^\infty$ of adapted left continuous processes of essentially bounded variation in a filtered probability space $(\Omega,\F,\FF,P)$. Here $\mu$ is a positive atomless optional random measure on $[0,T]$ and $h$ is a convex normal integrand on $\Omega\times[0,T]\times\reals^d$, and $k_0$ and $k_T$ are convex normal integrands on $\Omega\times\reals^d$. The main result of this paper gives an explicit expression for the conjugate of $Ef$ when the space $\N^\infty$ is identified with the Banach dual of {\em regular processes}, the optional projections of continuous processes with integrable supremum norm.

Under fairly general conditions, the domain of the conjugate of $Ef$ is contained in the space of regular quasimartingales which are semimartingales given by optional projections of continuous processes of integrable variation. This opens up the possibility of treating various problems in stochastic optimal control and mathematical finance with the theory of convex duality and integral functionals. Integral functionals of processes of bounded variation arise, for example, in problems of optimal investment under transaction costs and portfolio constraints. Such problems involve integral functionals of both the investment strategy as well as its derivative, which in general is a random measure. The basic theory of convex integral functionals of random measures have been developed in a companion paper \cite{pp15a}. Combining this with the main result of the present paper, allows for a unified treatment not only of optimal investment problems but also of more general problems in singular stochastic control much like \cite{bis73b} unified convex stochastic control problems without singularities. Singular stochastic control will be treated in a followup paper.

Our proofs are based on a version of the ``interchange rule'' that allows for reversing the order of minimization and integration in the optimization of integral functionals. In the case of decomposable spaces of measurable functions, such results go back to the works of \cite{roc68} and \cite{val75}. \cite{roc71} treated the nondecomposable space of continuous functions by embedding it in the space of essentially bounded measurable functions. \cite[Theorem~1]{bv88} gives a general interchange rule on spaces that are stable under continuous partitions of unity. This approach was used in \cite{per14} to study integral functionals of BV functions. In this paper we extend these arguments to the stochastic setting by combining them with interchange rules for Suslin space-valued functions from \cite{val75}.

The rest of this paper is organized as follows. Section~\ref{sec:det} studies integral functionals in the deterministic setting over the space $X$ of left-continuous functions of bounded variation. The section is split in two subsections, the first one giving an interchange rule for minimization and integration and the second on conjugates and subdifferentials when $X$ is regarded as the Banach dual of the space of continuous functions. Our main results are given in Section~\ref{sec:stoch} which follows a similar structure in the study of integral functionals over $\N^\infty$.

\section{Integral functionals of BV functions}\label{sec:det}

Given  a positive Radon measure $\mu$ on $[0,T]$ and $h$ a convex normal integrand on $\reals^d$ (see Appendix~\ref{app:if}), the associated integral functional on the space of measurable $\reals^d$-valued functions is defined by
\[
I_h(x):=\int h(x)d\mu:=\int_{[0,T]}h_t(x_t)d\mu_t.
\]
This section studies $I_h$ on the space $X$ of left-continuous functions $x:\reals\to\reals^d$ of bounded variation such that $x$ is constant outside of a fixed time interval $[0,T]$. Throughout this section, we assume that $\mu$ is atomless and has support $[0,T]$. Section~\ref{ssec:icrd} gives sufficient conditions for the interchange of the order of integration and minimization over $X$. Section~\ref{ssec:dualityd} uses the interchange rule to give an explicit expression for the conjugate of $I_h$ with respect to the pairing of $X$ with the space of $\reals^d$-valued continuous functions.

\subsection{Interchange rule}\label{ssec:icrd}

Recall that in a decomposable space $\X$ of measurable functions on $[0,T]$, one has the interchange rule
\[
\inf_{x\in\X} \int h(x)d\mu = \int\inf_{x\in\reals^d}h(x)d\mu
\]
as soon as the infimum on the left is finite; see \cite[Theorem~14.60]{rw98}. For this to hold when, instead of a decomposable space, one minimizes over $X$, we will need to control the behavior of the set
\[
\dom h_t=\{x\in\reals^d\mid h_t(x)<\infty\}
\]
as a function of $t$. Recall that a function is left-continuous in the usual sense if and only if it is continuous with respect to the topology $\tau$ generated by the left-open intervals $\{(s, t] \mid s<t\}$. Accordingly, a set-valued mapping $S:[0,T]\tos\reals^d$ is said to be {\em left-inner semicontinuous} (left-isc) if  $\{t \mid S_t\cap A\neq\emptyset\}$ is $\tau$-open for any open $A\subseteq\reals^d$. 

The following theorem is basically a reformulation of \cite[Theorem~4]{per14}. Given a measurable set-valued mapping $S$, we will use the notation
\[
L^\infty(S):=\{x\in L^\infty\,|\, x\in S\ \mu\text{-a.e.}\}
\]
and equip $L^\infty$ with the usual norm topology. We denote the interior of a set $A$ by $\interior A$.

\begin{theorem}\label{thm:BV}
Assume that $\dom h$ is left-isc, and that for every $x\in X$,
\[
x\in\interior L^\infty(\dom h) \implies x\in\dom I_h \implies x_t \in \cl\dom h_t\ \forall t. 
\]
If $I_h$ is proper on $X$ and $X\cap\interior L^\infty(\dom h)\ne\emptyset$, then
\[
\inf_{x\in X} \int h(x)d\mu= \int \inf_{x\in\reals^d}h(x)d\mu.
\]
\end{theorem}

\begin{proof}
By Theorems 4 and 2 of \cite{per14}, it suffices to show that every left-continuous $w$ with $w_t\in\cl\dom h_t$ $\mu$-a.e. satisfies $w_t\in\cl\dom h_t$ for all $t$. To have this, we can follow the arguments in the proof of \cite[Theorem 4]{per14} to get that, for every $\epsilon>0$ and $t$, there exists $x\in\dom I_h$ such that $|x_t-w_t|<\epsilon$. By assumption, such $x$ satisfies $x_t\in\cl\dom h_t$, and since $\epsilon>0$ was arbitrary, we have that $w_t\in\cl\dom h_t$ as well.
\end{proof}

When $I_h$ is lsc in $L^\infty$, the first condition 
\[
x\in\interior L^\infty(\dom h) \implies x\in\dom I_h
\]
in Theorem~\ref{thm:BV} is equivalent to $I_h$ being continuous at every $x\in \interior L^\infty(\dom h)$; see \cite[Corollary 8B]{roc74}. The second condition 
\[
x\in\dom I_h \implies x_t \in \cl\dom h_t\ \forall t
\]
in Theorem~\ref{thm:BV} holds in particular if $t\to \cl\dom h_t$ is {\em left-continuous} in the sense that it is left-isc and its graph $\{(t,x)\mid x\in\cl\dom h_t\}$ is closed in the product of $\tau$-topology on $[0,T]$ and the Euclidean topology on $\reals^d$.

\subsection{Duality}\label{ssec:dualityd}

The space $X$ may be identified with $\reals^d\times M$ where $M$ is the space of $\reals^d$-valued Radon measures on $[0,T]$. Indeed, given $x\in X$ there is a unique $\reals^d$-valued Radon measure $Dx$ on $[0,T]$ such that $x_t=x_0+Dx([0,t))$ for all $t\in[0,T]$ and $x_t=x_0+Dx([0,T])$ for $t>T$; see e.g.~\cite[Theorem~3.29]{fol99}. The value of $x\in X$ on $(T,\infty)$ will be denoted by $x_{T+}$. 

By the Riesz representation theorem, $M$ may be identified with the Banach dual of the space $C$ of continuous functions on $[0,T]$ when $C$ is equipped with the supremum norm. Indeed, $C$ and $M$ are in separating duality under the bilinear form
\[
\langle u,\theta\rangle :=\int ud\theta.
\]
Similarly, $X$ and $V:=\reals^d\times C$ are in separating duality under the bilinear form
\[
\langle v,x\rangle := v_{-\infty}\cdot x_{0}+\int vdx.
\]
Here, the notation corresponds to our convention of identifying elements of $V$ with continuous functions on $\{-\infty\}\cup[0,T]$ (Analogously, we may identify $X$ with the space of Radon measures on $\{-\infty\}\cup[0,T]$). 

\cite{roc71} and more recently \cite{per14} gave conditions under which the conjugate of an integral functional $I_h$ on $C$ can be expressed as
\[
(I_h)^*=J_{h^*},
\]
where, for a normal integrand $f$, the functional $J_f:M\to\ereals$ is defined by
\[
J_f(\theta)=\int f(d\theta^a/d\mu)d\mu+\int f^\infty(d\theta^s/d|\theta^s|)d|\theta^s|,
\]
where $\theta^a$ and $\theta^s$ are the absolutely continuous and the singular part, respectively, of $\theta$ with respect to $\mu$, $|\theta^s|$ is the total variation of $\theta^s$, and $f^\infty$ is the normal integrand defined pointwise as the recession function of $f_t$; see the appendix.

Theorem~\ref{thm:conjugated} below gives an expression for the conjugate and subdifferential of $I_h$ with respect to the pairing of $X$ with $V$. Given $x\in X$, we denote by $\partial h(x)$ the set-valued mapping $t\mapsto\partial h_t(x_t)$. We also use the notation $\partial^sh:=\partial\delta_{\cl\dom h}$. We will denote by $V_{BV}$ the linear subspace of $V$ consisting of $v\in V$ that have bounded variation on $[0,T]$. Given $v\in V_{BV}$, we set $dv/d\mu:=d(Dv)^a/d\mu$ and $dv/d|Dv^s|:=d(Dv)^s/d|(Dv)^s|$.

\begin{theorem}\label{thm:conjugated}
Let $f(x)=I_h(x)+k_0(x_0)+k_T(x_{T+})$, where $k_0$ and $k_T$ are closed proper convex functions on $\reals^d$ and $h$ is a convex normal integrand satisfying the assumptions of Theorem~\ref{thm:BV}. Then $f$ is closed and
\[
f^*(v)=
\begin{cases}
J_{h^*}(-Dv)+k_0^*(v_{-\infty}-v_0)+k_T^*(v_T) & \text{if $v\in V_{BV}$},\\
+\infty & \text{otherwise}.
\end{cases}
\]
Moreover, $v\in\partial f(x)$ if and only if 
\begin{align*}
-dv/d\mu &\in \partial h(x)\quad \mu\text{-a.e.},\\
-dv/d|Dv^s| &\in \partial^s h(x)\quad |(Dv)^s|\text{-a.e.},\\
v_{-\infty}-v_0 &\in\partial k_0(x_0),\\
v_{T} &\in\partial k_T(x_{T+}).
\end{align*}
\end{theorem}

\begin{proof}
This is a special case of Theorem~\ref{thm:ifn} below. The conjugate formula also follows as in the proof of \cite[Theorem 2.2]{pp14}.
\end{proof}



\section{Integral functionals of BV processes}\label{sec:stoch}

Let $(\Omega,F,P)$ be a probability space and let $\FF$ be an increasing sequence of $\sigma$-algebras on $\Omega$ that satisfies the usual hypotheses that $\F_t=\bigcap_{t'>t}\F_{t'}$ and $\F_0$ contains all the $P$-null sets. We will denote the linear space of left-continuous adapted processes of {\em essentially bounded variation} by $\N^\infty$. That is, $x\in\N^\infty$ if $x\in X$ almost surely, the pathwise total variation of $x$ is essentially bounded and $x_t$ is $\F_t$-measurable for all $t\in[0,T]$. 

Recall that the {\em predictable} and {\em optional} $\sigma$-algebras on $\Omega\times[0,T]$ are the ones generated by left- and right-continuous, respectively, adapted processes. In particular, the elements of $\N^\infty$ are predictable. We denote by $\T$ the set of {\em stopping times}, that is, functions $\tau:\Omega\to[0,T]\cup\{+\infty\}$ such that $\{(\omega,t)\,|\,\tau(\omega)\le t\}$ is optional. 
If a (not necessarily adapted) stochastic process $v$ is {\em $\T$-integrable} in the sense that $v_\tau$ is integrable for every $\tau\in\T$, then, e.g., by \cite[Theorem 5.1]{hwy92}, there exists an optional process $\op v$ such that
\begin{align*}
E[v_\tau\one_{\{\tau<\infty\}}\mid \F_\tau] &= \op{v}_\tau\one_{\{\tau<\infty\}}\quad P\text{-a.s. for all $\tau\in\T$},
\end{align*}
where $\F_\tau:=\{A\in\F\mid A\cap\{\tau\le t\}\in\F_t\ \forall t\in\reals_+\}$. The process $\op v$ is called the {\em optional projection} of $v$ and it is unique a.s.e.\footnote{The abbreviation a.s.e.\ stands for ``$P$-almost surely everywhere on $[0,T]$'', that is, outside an evanescent set.}. 

Recall that (see e.g.\ \cite[Theorem~3.16]{hwy92}) if $\tau$ is a stopping time then $[\tau,\infty)\subset\Omega\times\F$ is optional. A stopping time is said to be {\em predictable} if $[\tau,\infty)$ predictable. If a (not necessarily adapted) stochastic process $v$ is such that $v_\tau$ is integrable for every predictable time $\tau\in\T$, then, e.g., by \cite[Theorem 5.2]{hwy92}, there exists a predictable process $\pp v$ such that
\[
E[v_\tau\one_{\{\tau<\infty\}}\mid \F_{\tau-}] = \pp{v}_\tau\one_{\{\tau<\infty\}}\quad P\text{-a.s. for all predictable $\tau\in\T$},
\]
where $\F_{\tau-}:=\F_0\vee\sigma\{A\cap\{t<\tau\}\mid A\in\F_t,\, t\in\reals_+\}$. The process $\pp v$ is called the {\em predictable projection} of $v$ and it is unique a.s.e.

Let $\mu$ be a random positive measure on $[0,T]$ and let 
\[
\LL^p:=L^p(\Omega\times[0,T],\F\otimes\B([0,T]),\eta;\reals^d),
\]
where the measure $\eta$ is defined by $\eta(A):=E\int\one_Ad\mu$. We will assume throughout that $\mu$ is atomless, has full support almost surely and that it is {\em optional} in the sense that 
\[
E\int vd\mu = E\int \op vd\mu
\]
for all bounded $v$.

Let $h$ be a predictable normal integrand on $\reals^d$ and define $I_h:X\times\Omega\to\ereals$ by
\[
I_h(x,\omega):=I_{h(\cdot,\omega)}(x),
\]
where the right side is defined as in Section~\ref{sec:det}. We assume throughout that there exist $v\in\LL^1(\reals^d)$ and nonnegative $\alpha\in\LL^1$ such that
\begin{equation}\label{eq:lb}
h(x)\ge x\cdot v - \alpha.
\end{equation}
The following is proved in the appendix.

\begin{lemma}\label{lem:mif}
The function $I_h$ is a normal integrand on $X$. 
\end{lemma}
By Lemma~\ref{lem:mif}, the integral functional
\[
EI_h(x):=\int_\Omega I_h(x(\omega),\omega)dP(\omega)
\]
is a well-defined convex function on $\N^\infty$. Section~\ref{ssec:icr} below gives an interchange rule for $EI_h$ and Section~\ref{ssec:duality} gives an expression for the conjugate of $EI_h$ with respect to the pairing of $\N^\infty$ with regular processes to be defined in Section~\ref{ssec:banach}.

\subsection{Interchange rule}\label{ssec:icr}

The following result extends Theorem~\ref{thm:BV} to the stochastic setting. Given a measurable set-valued mapping $S$ from $\Omega\times[0,T]$ to $\reals^d$, we will use the notation
\[
\LL^\infty(S):=\{x\in \LL^\infty\,|\, x\in S\ \eta\text{-a.e.}\}
\]
and equip $\LL^\infty$ with the usual norm topology. We denote the closed unit ball with radius $r$ by $\uball_r$. 
\begin{theorem}\label{thm:ifn}
Assume that $h$ satisfies the assumptions of Theorem~\ref{thm:BV} almost surely and that there exists an $\bar x\in\N^\infty$ with $\bar x\in\inte\LL^\infty(\dom h)$ and that for every $x\in\interior\LL^\infty(\dom h)$ there is an $r>0$ and $\beta\in\LL^1$ with 
\[
h_t(x_t+x')\le\beta_t\quad \forall x'\in\uball_r.
\]
Then,
\begin{align*}
\inf_{x\in\N^\infty} E\int h(x)d\mu = E\int\inf_{x\in\reals^d}h(x)d\mu.
\end{align*}
\end{theorem}

\begin{proof}
By Lemma~\ref{lem:mif}, $I_h$ is a convex normal integrand on $X\times\Omega$. Applying first the interchange rule for expectation and minimization \cite[Theorem~1]{pp15a} and then the interchange rule Theorem~\ref{thm:BV}, we get
\begin{align*}
\inf_{x\in L^\infty(X)} E\int h(x)d\mu = E\inf_{x\in X}\int h(x)d\mu =E\int\inf_{x\in\reals^d}h(x)d\mu,
\end{align*}
where $L^\infty(X)$ is the space of (possibly nonadapted) left continuous processes of essentially bounded variation. On the other hand,
\begin{align*}
\inf_{x\in \N^\infty} EI_h(x) &\ge \inf_{x\in L^\infty(X)} EI_h(x)\ge \inf_{x\in L^\infty(X)} EI_h({^px}),
\end{align*}
where the second inequality follows from Jensen's inequality for predictable normal integrands; see Lemma~\ref{lem:jin} in the appendix. By \cite[Theorem VI.43]{dm82}, $\pp x$ is left continuous with right limits. 

We show next that the above infimum can be restricted to those $x\in L^\infty(X)$ for which $\uball(x_t,\tilde r)\subset\dom h_t$ for some $\tilde r>0$. If $x\in\dom EI_h$, then $x\in\dom I_h(x)$ a.s., so $x_t\in\cl\dom h_t$ a.s.e. We may assume that $\bar x$ satisfies $\bar x \in\dom h$, and we may redefine $x$ as $\bar x$ on a $P$-null set so that $x\in\cl\dom h$. Defining $x^\nu=\frac{1}{\nu}\bar x+(1-\frac{1}{\nu})x$, we have $x^\nu\in\interior\dom h$ for all $\nu$ and, by convexity, $EI_h(x^\nu)\le \frac{1}{\nu}EI_h(\bar x)+(1-\frac{1}{\nu})EI_h(x)$. Moreover, since $\uball(\bar x_t,r)\subset\dom h_t$ and $x_t\in\cl\dom h_t$, we have $\uball(x^\nu_t,r/\nu)\subset\dom h_t$ by convexity.

Thus, it suffices to show that, for every $\epsilon>0$ and $\tilde x\in L^{\infty}(X)$ with $\uball(\tilde x_t,\tilde r)\subset\dom h_t$ for some $\tilde r>0$, there exists $x\in\N^\infty$ such that 
\[
EI_h({^p \tilde x}) > EI_h(x) - \epsilon.
\]
Since $\dom h$ is predictable, it follows from Jensen's inequality for set-valued mappings \cite[Corollary 20]{kp15} that $\uball( {^p\tilde x}_t,\tilde r)\subset\dom h_t$. By \cite[Theorem 2]{roc71}, there exists $r'\in(0,\tilde r)$ such that
\[
 EI_h({^p \tilde x}) > EI_h(x) - \epsilon/2
\]
for every $x$ such that $|x_t-{^p\tilde x}_t|<r'$ a.s.e. For positive integers $\nu$, we define recursively $\tau^0=0$ and $\tau^\nu=\inf\{t\ge \tau^{\nu-1}\mid |{^p \tilde x}_t-{^p\tilde x}_{\tau^{\nu-1}}|\ge r'/2\}$ so that the process $\hat x=\sum_{\nu=0}^\infty \pp{\tilde x}_{\tau^\nu}\one_{(\tau^\nu,\tau^{\nu+1}]}$ is predictable with $|\hat x_t-{^p\tilde x}_t|<r'$. 

For every $n$, we define a stopping time $\sigma^n=\inf\{t\mid |D\hat x|([0,t])\ge n\}$ and a process $x^{n}=\one_{[0,\sigma^n]}\hat x+\one_{(\sigma^n,T]}\bar x$. By construction, $x^n\in\N^\infty$ and $x^n\rightarrow \hat x$ a.s.e. Moreover,
\[
h(x^n)\le \max\{\hat\beta,\bar\beta\},
\]
for some positive $\hat\beta$ and $\bar\beta$ with $E\int\max\{\hat\beta,\bar\beta\} d\mu<\infty$. Therefore, we may apply Fatou's lemma on $[0,T]\times\Omega$ to obtain 
\[
EI_h(\hat x)\ge EI_h(x^n)-\epsilon/2
\]
for $n$ large enough.
\end{proof}


\subsection{$\N^\infty$ as a Banach dual}\label{ssec:banach}


This section presents the basic functional analytic framework for our main result to be given in Section~\ref{ssec:duality}. In particular, we identify $\N^\infty$ as the topological dual of the Banach space of regular processes. We also recall some basic properties of quasimartingales that feature in the main result.

Let $L^\infty(M)$ be the linear space of {\em random $\reals^d$-valued Radon measures} with essentially bounded variation and let $\M^\infty\subseteq L^\infty(M)$ be the space of essentially bounded {\em optional} Radon measures on $\reals^d$, i.e.\ random measures $\theta\in L^\infty(M)$ such that
\[
E\int v d\theta = E\int \op{v}d\theta\quad\forall v\in L^1(C).
\]
Here $L^1(C)$ denotes the Banach space of random continuous functions $v$ with the norm
\[
\|v\|_{L^1(C)}:=E\sup_{t\in[0,T]}|v_t|.
\]
The linear mapping $x\mapsto(x_0,Dx)$ defines an isomorphism from $\N^\infty$ to $\reals^d\times\M^\infty$. By Theorem~\ref{thm:banach} below, $\M^\infty$ may be identified with the dual of the Banach space of regular processes, so $\N^\infty$ is a Banach dual as well.

Recall that a process is {\em regular} if it is the optional projection of a process from $L^1(C)$; see \cite{bis78}. 
We will denote the space of regular processes by $\R^1$. The following result, essentially proved already in Bismut~\cite{bis78}, is from \cite{pp15a}. 

\begin{theorem}\label{thm:banach}
The space $\R^1$ is a Banach space under the norm
\[
\|v\|_{\R^1} := \sup_{\tau\in\T}E|v_\tau|
\]
and its dual may be identified with $\M^\infty$ through the bilinear form	
\[
\langle v,\theta\rangle_{\R^1} := E\int vd\theta.
\]
The dual norm can be expressed as 
\[
\|\theta\|_{\M^\infty} = \esssup\|\theta\|_{TV}.
\]
\end{theorem}

Our main result, Theorem~\ref{thm:icr} below, involves regular processes which are also quasimartingales. Recall that a process $v$ is a {\em quasimartingale} if it is adapted, right continuous, has $E|v_t|<\infty$ for all $t$ and
\[
\Var(v):=\sup_\pi\left\{E\left[\sum_{i\in\pi} {|E\left [v_{t_{i+1}}-v_{t_i}\midb \F_{t_i}\right]|}\right]\right\}<\infty,
\]
where the supremum is taken over all finite partitions $\pi$ of $[0,T]$. The number $\Var(v)$ is known as the {\em mean variation} of $v$. Theorem~\ref{thm:qmart} below says in particular that quasimartingales are the optional projections of {\em IV processes}, i.e.\ (not necessarily) adapted processes whose initial value as well as pathwise total variation are integrable.

A process is said to be of {\em class $(D)$} if the set $\{v_\tau\,|\,\tau\in\T\}$ is uniformly integrable. In particular, regular processes are of class $(D)$. The following theorem, where $\R^1_m$ denotes the linear space of c\'adl\'ag martingales, summarizes some basic properties of quasimartingales.

\begin{theorem}\label{thm:qmart}
A c\'adl\'ag process $v$ is a quasimartingale of class $(D)$ if and only if $v=m+a$ for an $m\in\R^1_m$ and a predictable process $a$ of integrable variation with $a_0=0$. The decomposition is unique. One then has $\Var(v)=E\|Da\|_{TV}$ and
\[
E\left[\int_{[0,T]} vdx\right]=E\left[v_T\cdot x_{T+}-v_0\cdot x_0-\int_{[0,T]} x da\right]
\]
for every $x\in\N^\infty$. Quasimartingales of class $(D)$ are the optional projections of IV processes.
\end{theorem}

\begin{proof}
The first two claims are given Sections~3 and 4 of \cite[Appendix~2]{dm82}. The integration by parts formula follows from the first claim and the integration by parts formula for semimartingales (recalling that $x\in\N^\infty$ is left-continuous). 
If $v=m+a$, then it is the optional projection of the IV process $\one_{[0,T]} m_T+a$. Conversely, given an IV process $b$, we see from the definitions of optional projection and mean variation that $\op b$ has a finite mean variation. Moreover, the values of an IV process $b$ are bounded by $b_0+\|Db\|_{TV}$ so $b$ as well as $\op b$ are of class $(D)$.
\end{proof}

We will denote the linear space of regular quasimartingales by $\R^1_{qm}$.

\begin{corollary}\label{cor:qmart}
On $\R^1$, we have $\Var=\sigma_\D$, where 
\[
\D=\{Dx\in\M^\infty \mid x\in\N^\infty,\ |x|\le 1,\ x_0=x_{T+}=0\}.
\]
We have $v\in\R^1_{qm}$ if and only if $v=m+a$ for an $m\in\R^1_m$ and a predictable continuous process $a$ of integrable variation with $a_0=0$. Regular quasimartingales are the optional projections of continuous IV process.
\end{corollary}

\begin{proof}
By (3.4) in \cite[Appendix~II]{dm82}, $\Var\le\sigma_{\D}$. The opposite inequality follows from the expression $\Var(v)=E\|Da\|_{TV}$ and the integration by parts formula in Theorem~\ref{thm:qmart}. Given $v\in\R^1_{qm}$ and its decomposition in Theorem~\ref{thm:qmart}, we have $a=v-m\in\R^1$ since $\R^1_m\subset\R^1_{qm}$. Thus, by \cite[Remark VI.50(d)]{dm82}, $\pp a=a_-$ while $\pp a=a$ since $a$ is predictable. Hence, $a$ is continuous.
\end{proof}

\subsection{Duality}\label{ssec:duality}

By Theorem~\ref{thm:banach}, $\N^\infty$ may be identified with the Banach dual of the space $\V^1:=\reals^d\times\R^1$ under the bilinear form
\[
\langle x,v\rangle_{\N^\infty} = E\left[x_0\cdot v_{-\infty} + \int vdx\right],
\]
where we regard elements of $\V^1$ as regular processes  on $\{-\infty\}\cup[0,T]$. Our main result, Theorem~\ref{thm:ifn} below, gives an explicit expression for the conjugate of an integral functional on $\N^\infty$ with respect to the above pairing. 




Let $h$ be as in Section~\ref{ssec:icr} and define
\[
f(x,\omega) = I_h(x,\omega) + k_0(x_0,\omega) + k_T(x_{T+},\omega),
\]
where $k_0$ and $k_T$ are convex normal $\F_0$- and $\F_T$-integrands on $\reals^d$, respectively. Here all the three terms define normal integrands, so $f$ is a normal integrand as well; see \cite[Lemma~24]{pp15a}. We assume throughout that $Ek_0$ and $Ek_T$ are proper on $L^\infty(\reals^d)$ that $Ek_0^*$ and $Ek_T^*$ are proper on $L^1(\reals^d)$.

Given $v\in\R^1_{qm}$ and its unique decomposition $v=m+a$ in Corollary~\ref{cor:qmart}, we denote $Dv:=Da$ and continue to use the notation $dv/d\mu:=d(Dv)^a/d\mu$ and $dv/d|Dv^s|:=d(Dv)^s/d|(Dv)^s|$ from Section~\ref{ssec:dualityd}. We also define $\V^1_{qm}:=\reals^d\times\R^1_{qm}$.

\begin{theorem}\label{thm:icr}
Let $h$ satisfy the assumptions of Theorem~\ref{thm:ifn}. Then $Ef:\N^\infty\to\ereals$ is closed and its conjugate can be expressed as
\[
(Ef)^*(v)=
\begin{cases}
E\left[J_{h^*}(-Dv)+k_0^*(v_{-\infty}-v_0)+k^*_T(v_T)\right] & \text{if $v\in\V^1_{qm}$},\\
+\infty & \text{otherwise.}
\end{cases}
\]
Moreover, $v\in\partial Ef(x)$ if and only if
\begin{align*}
-dv/d\mu &\in \partial h(x)\quad \mu\text{-a.e.},\\
-dv/d|Dv^s| &\in \partial^s h(x)\quad |(Dv)^s|\text{-a.e.},\\
v_{-\infty}-v_0 &\in\partial k_0(x_0),\\
v_{T} &\in\partial k_T(x_{T+})
\end{align*}
almost surely.
\end{theorem}

\begin{proof} 
By the assumptions of Theorem~\ref{thm:ifn}, there exist $\alpha$ and $r>0$ such that $Ef(\bar x+x)\le\alpha$ for all $x\in\N^\infty$ with $x_0=x_{T+}=0$ and $|x_t(\omega)|\le r$ a.s.e. Thus,
\begin{align*}
(Ef)^*(v)&=\sup_{x\in\N^\infty}\{\langle x,v\rangle-Ef(x)\}\\
&\ge\sup_{x\in\N^\infty}\left\{\langle \bar x+ x,v\rangle-Ef(\bar x+x)\midb |x_t|\le r,\ x_0=x_{T+}=0\right\}\\
&\ge\sup_{x\in\N^\infty}\left\{\langle \bar x+ x,v\rangle-\alpha\midb |x_t|\le r,\ x_0=x_{T+}=0\right\}\\
&\ge \langle \bar x,v\rangle -\alpha + \sup_{x\in\N^\infty}\left\{\langle Dx,v\rangle\midb |x_t|\le r,\ x_0=x_{T+}=0\right\}\\
&= \langle \bar x,v\rangle -\alpha + r\Var(v),
\end{align*}
where the last equality follows from Corollary~\ref{cor:qmart}. When $\Var(v)<\infty$, the integration by parts formula in Theorem~\ref{thm:qmart} gives
\begin{align*}
(Ef)^*(v)&=\sup_{x\in\N^\infty}E[\langle v,x\rangle - f(x)]\\
&=\sup_{x\in\N^\infty}E\left[v_T\cdot x_{T+}+(v_{-\infty}-v_0)\cdot x_0-\int xda - f(x)\right]\\
&=\sup_{x\in\N^\infty}E\left[-\int  xda-I_h(x)\right]+Ek_0^*(v_{-\infty}-v_0)+Ek^*_T(v_T).
\end{align*}
Here the last equality follows by first noting that $a$ and $\mu$ do not have atoms at the origin and that $Ek_0$ and $Ek_T$ are proper so that we may take expectation separately from each term, and then by applying the interchange rule for minimization and expectation \cite[Theorem 14.60]{rw98}. It thus suffices to show that
\begin{align*}
\sup_{x\in \N^\infty} E\left[-\int xda - I_h(x)\right]=E J_{h^*}(-Dv).
\end{align*}

Using \cite[Theorem 5.15]{hwy92}, there exists a predictable set $A$ such that $E\int 1_{A^C} (dv/d\mu) d\mu=E\int 1_{A} d|(Dv)^s|=0$. Defining $\bar\mu=|(Dv)^s|+\mu$ and 
\[
\bar h_t(x,\omega)=\begin{cases}
 h_t(x,\omega) + x\cdot (dv/d\mu)_t(\omega)\quad&\text{if }(\omega,t)\in A,\\
 \delta_{\cl\dom h_t(\omega)} (x) + x\cdot (dv/d|Dv^s|)_t(\omega)\quad&\text{otherwise},
\end{cases}
\]
we have, by the last assumption in Theorem~\ref{thm:BV}, that
\[
E\left[I_h(x)+\int xda \right]=E\int \bar h(x)d\bar\mu.
\]
We have
\[
\inf_{x\in\reals^d}\bar h_t(x,\omega)=\begin{cases}
 -h^*_t(-(dv/d\mu)_t(\omega),\omega)\quad&\text{if }(\omega,t)\in A,\\
 - (h_t^*)^\infty(-(dv/d|Dv^s|)_t(\omega)\quad&\text{otherwise}
\end{cases}
\]
while it is straight-forward to verify the assumptions in Theorem~\ref{thm:ifn}, so 
\[
\inf_{x\in\N^\infty}E\int \bar h(x)d\bar\mu = E\int \inf_{x\in\reals^d}\bar h(x)d\bar\mu =-J_{h^*}(-Dv).
\]

To prove the subgradient formula, let $x\in\dom Ef$ and $v\in\V^1_{qm}$. By Fenchel inequality,
\begin{align*}
h(v)+h^*(-dv/d\mu)&\ge -x\cdot(dv/d\mu)\quad\mu\text{-a.e.,}\\
(h^*)^\infty(-dv/d|Dv^s|) &\ge -x\cdot (dv/d|Dv^s|)\quad|(Dv)^s|\text{-a.e.,}\\
k_0(x_0)+k^*_0(v_{-\infty}-v_0)  &\ge x_0\cdot (v_{-\infty}-v_0),\\
k_T(x_{T+})+k_T^*(v_T) &\ge x_{T+}\cdot v_T
\end{align*}
almost surely. By the definition of a subgradient, $v\in\partial Ef(x)$ if and only if $Ef(x)+(Ef)^*(v)=\langle v,x\rangle$. Using the first part of the theorem and the integration parts formula in Theorem~\ref{thm:qmart}, we see that this is equivalent to having the above inequalities satisfied as equalities which in turn is equivalent to the stated pointwise subdifferential conditions.
\end{proof}

We say that a normal integrand $h$ is {\em integrable} if $h(x)\in \LL^1$ for every constant process $x$. For real-valued and integrable $h$, Theorem~\ref{thm:ifn} takes a simpler form.

\begin{corollary}\label{cor:inth}
If $h$ is real-valued and integrable, then the conjugate of $Ef:\N^\infty\to\ereals$ can be expressed as
\[
(Ef)^*(v)=
\begin{cases}
E\left[J_{h^*}(-Dv)+k_0^*(v_{-\infty}-v_0)+k^*_T(v_T)\right] & \text{if $v\in\V^1_{qm}$ and $Dv\ll\mu$},\\
+\infty & \text{otherwise}
\end{cases}
\]
and $v\in\partial Ef(x)$ if and only if $Dv\ll \mu$ and
\begin{align*}
-dv/d\mu &\in \partial h(x)\ \mu\text{-a.e.},\\
v_{-\infty}-v_0 &\in\partial k_0(x_0),\\
v_{T} &\in\partial k_T(x_{T+})
\end{align*}
almost surely.
\end{corollary}
\begin{proof}
Inspection of the proof of \cite[Theorem 2]{roc71} reveals that, integrability of $h$ implies that for every $x\in \LL^\infty$ and $r>0$, there exists $\beta\in\LL^1$ such that 
\[
h(x+x')\le \beta\quad\forall x'\in\uball_r.
\]
Thus, since $\dom h_t=\reals^d$ for all $t$, $f$ satisfies the assumptions of Theorem~\ref{thm:icr} and $\partial^sh_t=\{0\}$ for all $t$ almost surely.
\end{proof}

In the opposite extreme where, instead of a real-valued function, $h$ is the indicator function of a random set, Theorem~\ref{thm:icr} takes the following form.

\begin{corollary}\label{cor:ifind}
Let $S$ be a predictable closed convex-valued mapping and
\[
f(x,\omega)= \begin{cases} 
k_0(x_0,\omega)+k_T(x_{T+},\omega)\quad&\text{if }x_t\in S_t(\omega)\ \forall t, \\
+\infty\quad&\text{otherwise}
\end{cases}
\]
Assume that there exists an $\bar x\in\dom Ef\cap\N^\infty$ with $\bar x\in\inte\LL^\infty(S)$. Then the conjugate of $Ef:\N^\infty\to\ereals$ can be expressed as
\[
(Ef)^*(v)=
\begin{cases}
E\left[\int \sigma_S(-d(Dv)/d|Dv|)d|Dv|+k_0^*(v_{-\infty}-v_0)+k^*_T(v_T)\right] & \text{if $v\in\V^1_{qm}$},\\
+\infty & \text{otherwise.}
\end{cases}
\]
Moreover, $v\in\partial Ef(x)$ if and only if
\begin{align*}
-d(Dv)/d|Dv| &\in N_S(x)\ |Dv|\text{-a.e.},\\
v_{-\infty}-v_0 &\in\partial k_0(x_0),\\
v_{T} &\in\partial k_T(x_{T+})
\end{align*}
almost surely.
\end{corollary}

Note that when $S$ is the unit ball and $k_0=k_{T+}=\delta_{\{0\}}$, Corollary~\ref{cor:ifind} gives the expression $\Var=\sigma_\D$ from Corollary~\ref{cor:qmart}.

\begin{corollary}
Let $f(x,\omega) = k_0(x_0,\omega) + k_T(x_{T+},\omega)$. If $Ef:\N^\infty\to\ereals$ is proper, then
\[
(Ef)^*(v)=
\begin{cases}
E\left[k_0^*(v_{-\infty}-v_0)+k^*_T(v_{T})\right] & \text{if $v\in\R^1_m$},\\
+\infty & \text{otherwise.}
\end{cases}
\]
Moreover, $v\in\partial Ef$ if and only if $v\in\R^1_m$ and
\begin{align*}
v_{-\infty}-v_0 &\in\partial k_0(x_0),\\
v_{T} &\in\partial k_T(x_{T+})
\end{align*}
almost surely.
\end{corollary}


The {\em Mackey} topology on a locally convex vector space $U$, is the convex topology generated by the level sets of support functions of weakly-compact sets on the dual space $Y$. It is denoted by $\tau(U,Y)$. By the Mackey--Arens theorem, $\tau(U,Y)$ is the strongest locally convex topology under which every continuous linear functional can be represented as $u\mapsto\langle u,y\rangle$ for some $y\in Y$. A version of Alaoglu's theorem states that if a convex function on $U$ is Mackey continuous, then the level sets of its conjugate are $\sigma(Y,U)$-compact; see \cite[Theorem~10b]{roc74}.

\begin{theorem}\label{thm:mackey}
If $h$, $k_0$ and $k_T$ are real-valued and integrable, then $Ef$ is Mackey-continuous and, in particular, $(Ef)^*$ has compact level-sets in $\V^1$.
\end{theorem}

\begin{proof}
By \cite[Theorem 22]{roc74}, $EI_h$ is $\tau(\LL^\infty,\LL^1)$-continuous on $\LL^\infty$ while $Ek_0$ and $Ek_T$ are $\tau(L^\infty,L^1)$-continuous on $L^\infty$. By Lemma~\ref{lem:adj}, it suffices to show that the embedding $i$ of $\N^\infty$ is weakly continuous on $\LL^\infty$ and that $x\mapsto x_0$ and $x\mapsto x_{T_+}$ are weakly continuous from $\N^\infty$ to $L^\infty$.

Let $w\in\LL^1$ and define $z\in L^1(C)$ by $z_t=\int_{[0,t]}wd\mu$. Integration by parts gives
\begin{align*}
\langle i(x),w\rangle_{\LL^\infty} &= E\int x dz\\
&= E\left[x_{T+}\cdot z_T-\int zdx\right]\\
&= E\left[x_0\cdot z_T+\int (z_T-z)dx\right]\\
&= E\left[x_0\cdot E_0z_T+\int\op(z_T-z)dx\right]\\
&=\langle x,v\rangle_{\V^1},
\end{align*}
where $v=(E_0z_T,\op(z_T-z))\in\V^1$. Thus, the adjoint $i^*:\LL^1\to\V^1$ of $i$ has full domain and is given by $i^*(w)=(E_0z_T,\op(z_T-z))$. This implies the weak continuity of both $i$ and $i^*$. The continuity of $x\mapsto x_T$ follows from
\[
E(x_T\cdot z)=E\left(x_0\cdot z+\int zdx\right)=E\left(x_0\cdot E_0z+\int\op zdx\right)
\]
and that of $x\mapsto x_0$ from $E(x_0\cdot z)= E(x_0\cdot E_0z)$.
\end{proof}

We end this section by studying asymptotic properties of integral functionals and dense subsets of $\V^1$. 

\begin{theorem}\label{thm:rec}
If $Ef$ is proper, then 
\[
(Ef)^\infty(x)=E\left[I_{h^\infty}(x)+k_0^\infty(x_0)+k_T^\infty(x_{T+})\right].
\]
\end{theorem}

\begin{proof}
Like in the proof of Theorem~\ref{thm:mackey}, we may view $Ef$ as the sum of three functions, each one of which is the composition of a continuous linear mapping and an integral functional. The lower bound \eqref{eq:lb} implies that $EI_h$ is lsc on $\LL^\infty$, and the properness assumptions of $Ek_0^*$ and $Ek_T^*$ imply that $Ek_0$ and $Ek_T$ are lsc on $L^\infty$. It now suffices to apply the last part of \cite[Theorem~2]{pp15a} to the integral functionals and to use the general facts that the recession function of a sum/composition is the sum/composition of the recession functions whenever the sum/composition is proper.
\end{proof}

Following \cite[Section~3D]{rw98}, we say that a normal integrand $h$ is {\em coercive} if $h^\infty=\delta_{\{0\}}$.

\begin{corollary}\label{cor:ito}
If $h$, $k_0$ and $k_T$ are coercive, then $\dom(Ef)^*$ is dense in $\V^1$. In particular, $\{v\in\V^1_{qm}\,|\,Dv\ll\mu\}$ is dense in $\V^1$.
\end{corollary}

\begin{proof}
By Theorem~\ref{thm:rec}, $Ef$ is coercive on $\N^\infty$. By \cite[Theorem 5.B]{roc66}, this is equivalent to $\dom (Ef)^*$ being dense in $\V^1$. The last claim follows e.g.\ by taking $h$, $k_0$ and $k_T$ quadratic.
\end{proof}



The above implies, in particular, that It{\^o} processes are dense in $\R^1$. 

\begin{remark}
Our results are easily specialized to functionals of the form
\[
Ef_0(x):= E\left[I_h(x)+k_{T+}(x_{T+})\right]
\]
on the space $\N_0^\infty:=\{x\in \N^\infty\mid x_0=0\}$. Indeed, we may pair $\N_0^\infty$ with $\R^1$ via 
\[
\langle x,v\rangle=E\int vdx,
\]
the weak topologies of $\N_0^\infty$ and $\R^1$ are simply the relative topologies weak topologies when $\N_0^\infty$ and $\R^1$ are viewed as subspaces of $\N^\infty$ and $\V^1$, respectively. Setting $k_0=\delta_0$, we have
\[
(Ef_0)^*(v)=(Ef)^*(0,v).
\]
\end{remark}


\section{Appendix}

This section recalls some basic definitions and facts from convex duality and the theory of integral functionals.

\subsection{Duality}

When $U$ is in separating duality with another linear space $Y$, the {\em conjugate} of an extended real-valued convex function $g$ on $U$ is the extended real-valued function $g^*$ on $Y$ defined by
\[
g^*(y) = \sup_{u\in U}\{\langle u,y\rangle - g(u)\}.
\]
A $y\in Y$ is a {\em subgradient} of $g$ at $u$ if
\[
g(u')\ge g(u) + \langle u'-u,y\rangle\quad\forall u'\in U.
\]
The set $\partial g(u)$ of all subgradients is known as the {\em subdifferential} of $g$ at~$u$. We often use the fact $y\in\partial g(u)$ if and only if
\[
g(u)+g^*(y)=\langle u,y\rangle.
\]
The {\em recession function} of a closed proper convex function $g$ is defined by
\[
g^\infty(u)=\sup_{\alpha>0}\frac{g(\bar u+\alpha u)-g(\bar u)}{\alpha},
\]
where the supremum is independent of the choice of $\bar u\in\dom g$; see \cite[Corollary~3C]{roc66}. By \cite[Corollary 3D]{roc66}, $\delta_{\dom g^*}^*=g^\infty$.

\begin{lemma}\label{lem:adj}
If $A:X\to U$ is a weakly continuous linear mapping with respect to the pairings of $X$ with $V$ and $U$ with $Y$, then $A$ is Mackey-continuous.
\end{lemma}

\begin{proof}
If $O\subset\tau(U,Y)$ is a neighborhood of the origin, there exists a weakly compact $D\subset Y$ with $\lev_1\sigma_D\subset O$ so
\[
A^{-1}(O)\supset\{x\in\N^\infty\,|\,\sigma_D(Ax)\le 1\}=\{x\in\N^\infty\,|\,\sigma_{A^*D}(x)\le 1\}.
\]
Weak continuity of $A$ implies that it has a weakly continuous adjoint $A^*$, so $A^*D\subset V$ is weakly compact. 

\end{proof}

\subsection{Integral functionals}\label{app:if}

Given a measurable space $(\Xi,\A)$ and a locally convex topological vector space $U$, a set-valued mapping $S:\Xi\tos U$ is {\em measurable} if the inverse image $S^{-1}(O):=\{\xi\in\Xi\,|\,S(\xi)\cap O\ne\emptyset\}$ of every open $O\subseteq S$ is in $\A$. An extended real-valued function $f:U\times\Xi\to\ereals$ is said to be a {\em normal $\A$-integrand} on $\reals^d$ if the {\em epi-graphical mapping}
\[
\xi\mapsto\epi f(\cdot,\xi)=\{(u,\alpha)\in U\times\reals|\, f(u,\xi)\leq\alpha\}
\]
is closed-valued and measurable. A normal integrand $f$ is said to be {\em convex} if $f(\cdot,\xi)$ is a convex function for every $\xi\in\Xi$. When $U$ is a Suslin space as well as a countable union of Borel sets that are Polish in the relative topology, a normal integrand is $\B(U)\otimes\A$-measurable (see \cite{pp15a}), so $\xi\mapsto f(u(\xi),\xi)$ is $\A$-measurable whenever $u:\Xi\to U$ is $\A$-measurable. Given a nonnegative measure $\eta$ on $(\Xi,\A)$, the measurability implies that the associated {\em integral functional}
\[
I_f(u):=\int_\Xi f(u(\xi),\xi)d\eta(\xi)
\]
is a well-defined extended real-valued function on the space $L^0(\Xi,\A,\eta;U)$ of equivalence classes of $U$-valued $\A$-measurable functions. Here and in what follows, we define the integral of a measurable function as $+\infty$ unless the positive part of the function is integrable. The function $I_f$ is called the {\em integral functional} associated with the normal integrand $f$. If $f$ is a convex normal integrand, $I_f$ is convex on $L^0(\Xi,\A,\eta;U)$.

\begin{lemma}[Jensen's inequality]\label{lem:jin}
Assume that $h$ is an optional convex normal integrand, $\mu$ is an optional random measure and that
\[
h(x)\ge x\cdot v-\alpha
\] 
for some optional $v$ and nonnegative $\alpha$ such that $\int |v|d\mu$ and $\int\alpha d\mu$ are integrable. Then
\[
EI_h(x)\ge EI_h(\op x)
\]
for every bounded process $x$. If $h$, $\mu$ and $v$ are predictable, then
\[
EI_h(x)\ge EI_h(\pp x)
\]
for every bounded process $x$.
\end{lemma}
\begin{proof}
We define $\hat\mu\ll \mu$ by $d\hat\mu/d\mu=\beta:=\op(1/(1+\int d\mu))$ so that $\hat\mu$ defines an optional bounded measure $\hat \eta(A)=E\int\one_A d\hat\mu$ on $\Omega\times[0,T]$. Moreover, $EI_h(x)=E\int \hat h(x)d\hat\mu$, where $\hat h(x)=h(x)/\beta$ is an optional convex normal integrand. We have
\[
\hat h^*(v)=h^*(\beta v)/\beta,
\]
so the lower bound implies that $E\int \hat h^*(v/\beta)d\hat\mu$ is finite. Thus we may apply the interchange of integration and minimization on $(\Omega\times[0,T],\O,\hat\eta)$ and on $(\Omega\times[0,T],\F\otimes\B([0,T]),\hat\eta)$ (see \cite[Theorem 14.60]{rw98}) to get
\begin{align*}
EI_h(\op x) &=E\int \hat h(\op x)d\hat\mu\\
 &=\sup_{v\in \LL^1(\Omega\times[0,T],\O,\hat\eta)}E\int[\op x\cdot v-\hat h^*(v)]d\hat\mu\\
&=\sup_{v\in \LL^1(\Omega\times[0,T],\O,\hat\eta)}E\int[x\cdot v-\hat h^*(v)]d\hat\mu\\
&\le \sup_{v\in \LL^1(\Omega\times[0,T],\F\otimes\B([0,T]),\hat\eta)}E\int[x\cdot v-\hat h^*(v)]d\hat\mu\\
&=E\int \hat h(x)d\hat\mu\\
&=EI_h(x).
\end{align*}
The predictable case is proved similarly.
\end{proof}

\subsection{Proof of Lemma~\ref{lem:mif}}
By \cite[Lemma~22]{pp15a}, it suffices to show that $I_h(\cdot,\omega)$ is lsc almost surely and that $I_h$ is $\F\otimes\B(X)$-measurable.

To show that $I_h(\cdot,\omega)$ is lsc almost surely, we denote 
\[
L_\omega^\infty :=L^\infty([0,T],\B([0,T]),\mu(\omega);\reals^d).
\]
Since $\mu(\omega)$ is atomless, the embedding of $(X,\sigma(X,V))$ to $(L_\omega^\infty,\sigma(L_\omega^\infty,L_\omega^1)$ is continuous (see the proof of \cite[Theorem~2.1]{pp14}) while the lower bound implies that $I_h(\cdot,\omega)$ is $\sigma(L_\omega^\infty,L_\omega^1)$-lsc, by \cite[Theorem~3C]{roc76}.

To prove the measurability, let $S$ be the space of l\'adc\'ag functions (left-continuous with right limits) from $[0,T]$ to $\reals^d$. Equipped with the Skorokhod topology (with obvious changes of signs since we deal with left-continuous instead of right continuous functions) $S$ is a Polish space; see \cite[Theorem 15.17]{hwy92}. Since every sequence converging in $S$ converges pointwise outside a countable set\footnote{A l\'adc\'ag function has at most a countable set of discontinuities, so this fact follows from the remark on page 452 in \cite{hwy92}}, $S$ satisfies the assumptions of \cite[Theorem~25]{pp15a}, so $I_h$ is $\B(S)\otimes\F$-measurable. It thus suffices to show that the injection from $X$ to $S$ is measurable.

By \cite[Theorem 3]{pes95}, $\B(S)$ coincides with the Borel-$\sigma$-algebra generated by the topology that $S$ has when equipped with the supremum norm. By \cite[Theorem VII.65]{dm82}, continuous linear functionals in the weak topology are of the form
\[
l(u) := \int u_tda_t+\int u_{t+}da^+_t,
\]
where $a$ and $a^+$ are functions of bounded variation. When $u\in X$, integration by parts gives
\[
l(u)= u_{T+}(a_{T+}+a^+_{T+})-u_0(a_{0}+a^+_{0})-\int a_t du_t -\int a_{t-}^+ du_t.
\]
Since every function of bounded variation is a pointwise limit of a sequence of continuous functions, it is not difficult to verify that $l$ is measurable in $\B(X)$.\qed

\bibliographystyle{alpha}
\bibliography{sp}

\end{document}